\documentclass[12pt]{article}
\usepackage{verbatim}
\usepackage{amssymb}
\usepackage{amsthm}
\usepackage{amsmath}
\usepackage{mathrsfs}
\usepackage[margin=1in]{geometry}

\usepackage{latexsym}
\usepackage{amssymb}
\usepackage{euscript}

\let\cal=\mathcal      

\def\mcc{M\raise.5ex\hbox{c}C}
\def\mccarthy{M\raise.5ex\hbox{c}Carthy}

\def\h{{\cal H}}

\def\K{{\cal K}}
\def\M{{\cal M}}

\def\m{Mult}

\let\i=\infty

\def\la{\langle}
\def\ra{\rangle}
\def\={\ = \ }

\def\C{\mathbb C}

\def\D{\mathbb D}

\def\be{\setcounter{equation}{\value{theorem}} \begin{equation}}
\def\ee{\end{equation} \addtocounter{theorem}{1}}
\def\beq{\begin{eqnarray*}}
\def\eeq{\end{eqnarray*}}
\def\se{\setcounter{equation}{\value{theorem}}} 
\def\att{\addtocounter{theorem}{1}}
\def\vs{\vskip 5pt}

\def\bp{{\sc Proof: }}
\def\ep{{}{\hfill $\Box$} \vskip 5pt \par}

\def\bl{\begin{lemma}}
\def\el{\end{lemma}}
\def\bt{\begin{theorem}}
\def\et{\end{theorem}}
\def\bprop{\begin{prop}}
\def\eprop{\end{prop}}
\def\bd{\begin{definition}}
\def\ed{\end{definition}}
\def\br{\begin{remark}}
\def\er{\end{remark}}
\def\bexer{\begin{exercise}}
\def\eexer{\end{exercise}}

\newtheorem{theorem}{Theorem}[section]

\newtheorem{lemma}[theorem]{Lemma}
\newtheorem{corollary}[theorem]{Corollary}

\newtheorem{definition}[theorem]{Definition}

\usepackage[usenames,dvipsnames] {color}

\definecolor{dark_purple}{rgb}{0.4, 0.0, 0.4}
\definecolor{dark_green}{rgb}{0.0, 0.7, 0.0}

\newtheorem{defin}[theorem]{Definition}
\newtheorem{prop}[theorem]{Proposition}
\newtheorem{lem}[theorem]{Lemma}
\numberwithin{equation}{section}

\title{Beurling's theorem for the Hardy operator on $L^2[0,1]$}
\author{Jim Agler
and
John E. M\raise.5ex\hbox{c}Carthy
\thanks{Partially supported by National Science Foundation Grant  
DMS 2054199}
}
\date{\today}

\newcommand{\mc}{M\raise.45ex\hbox{c}Carthy}
\renewcommand\Re{\mathrm {Re~}}

\def\norm#1{\| #1 \|}

\newcommand\s{{\mathbb S}}
\newcommand\ltwo{L^2}

\usepackage{mathtools}

\renewcommand\m{{\mathcal M}}
\newcommand\n{{\mathcal N}}

\begin{document}

\bibliographystyle{plain}

\maketitle

\begin{abstract}
We prove that the invariant subspaces of the Hardy operator on $L^2[0,1]$ are the spaces that are limits of sequences of finite dimensional spaces spanned by monomial functions.
\end{abstract}

\section{Introduction}
\label{seca}

The space $L^2[0,1]$ is a cornerstone of analysis. 
One way to analyze it is to use the exponential functions $e^{itx}$, which have the advantage of being eigenfunctions for differentiation.
 Another way is to use the monomial functions $x^s$. The M\"untz-Sz\'asz theorem gives necessary and sufficient conditions for a collection of monomial functions to span $L^2[0,1]$. 
Monomials are eigenfunctions for the Hardy operator $H$, 
 defined by
\[
H f(x) \= \frac{1}{x} \int_0^x f(t) dt .
\]
Conversely, if $T$ is a bounded linear operator on $L^2[0,1]$ that has $x^s$ as an eigenvector whenever $x^s$ is in $L^2[0,1]$, then $T$ is a function of $H$; specifically, it is of the form $\phi(H)$ for some function $\phi$ that is bounded and analytic on the disk $\D(1,1) = \{ z \in \C : |z-1| < 1 \}$ \cite{amMO}.

We shall use $L^2$ to denote $L^2[0,1]$ throughout. 
Hardy proved in \cite{har20} that $H$ is bounded on $L^2$  (and indeed on $L^p$ for all $p > 1$).
For a treatment of $H$ consult the book \cite{kmp07}.
What are its invariant subspaces?

Let $\s$ denote the half plane $\{ s \in \C : \Re(s) > - \frac12 \}$. Then if $s \in \s$, the monomial function $x^s$ is in $L^2$,
and $H x^s \= \frac{1}{s+1} x^s$; moreover the monomials constitute all the eigenvectors of $H$. Any space that is
the linear span of finitely many monomial functions is invariant for $H$. We shall call such a space a finite monomial space.
It is the object of this note to prove that every invariant subspace of $H$ is a limit of  finite monomial spaces. 

The Hardy operator is unitarily equivalent to $1- S^*$, where $S$ is the unilateral shift \cite{bhs65}. Its invariant subspaces are therefore described by the celebrated theorem of Beurling \cite{beu} which described the invariant subspaces of the shift using the beautiful theory of Hardy spaces of holomorphic functions. Using this theory, Theorem \ref{thmis} below is well-known. It is proved as the Theorem on Finite Dimensional Approximation \cite[p.37]{nik}. However, the point of this note is to describe the invariant subspaces of $H$ without using any Hardy space theory, just using $L^2$ techniques and functional analysis.
Our hope is that this approach will not only illuminate $L^2$ with a new light, but may also generalize to related spaces, such as $L^p$ or weighted $L^p$ spaces.

\begin{defin}
For $S$ a finite subset of $\s$ we let $\m(S)$ denote the span in $\ltwo$ of the monomials whose exponents lie in $S$, i.e.,
\[
\m(S)=\{\sum_{s\in S}a(s)x^s\  |\ a:S\to \C \}.
\]
We refer to sets in $\ltwo$ that have the form $\m(S)$ for some finite subset $S$ of $\s$ as \emph{finite monomial spaces}.\end{defin}

\begin{defin}
\label{defin1}
If $\m$ is a subspace of a Hilbert space $\h$  and $\{\m_n\}$ is a sequence of closed subspaces, we say that \emph{$\{\m_n\}$ tends to $\m$} and write
\[
\m_n \to \m\ \text{ as }\ n\to \infty
\]
if
\[
\m = \{f \in \h\ | \ \lim_{n \to \infty} {\rm dist}(f,\m_n)=0\}.
\]
\end{defin}

\begin{defin}
We say that a subspace $\m$ of $\ltwo$ is a \emph{monomial space} if there  exists a sequence $\{\m_n\}$ of finite monomial spaces such that $\m_n \to \m$.
\end{defin}
Equipped with these definitions, we can now state our main theorem.
\bt
\label{thmis}
Let $\m$ be a closed non-zero subspace of $\ltwo$. Then $\m$ is invariant for $H$ if and only if $\m$ is a monomial space.
\et
One way to construct a monomial space is to take the closed linear span of an infinite set of monomial functions,
\be
\label{eqa1}
\m \= \vee \{ x^{s_k} : k \in {\mathbb N} \}.
\ee
The M\"untz-Sz\'asz  theorem (proved in \cite{mu14,sza16} for integer exponents, and in \cite{sza53} for general real exponents)
characterizes when such a space is a proper subspace of $\ltwo$. See \cite{boer95} for a thorough treatment.
\bt (M\"untz-Sz\'asz)
\label{thmms}
\[
 \vee \{ x^{s_k} : k \in {\mathbb N}  \} \= \ltwo\ \ \  \text{ if and only if }\ \ \ \sum_k \frac{2 \Re s_k +1}{|s_k+ 1|^2 } =\infty.
\]
\et
Not every monomial space looks like \eqref{eqa1}. It is easy to see that for any $0 < s < 1$, the space
$\{ f \in \ltwo : f = 0 {\rm\ a.e.\ on\ } [0,s] \}$ is invariant for $H$, and hence is a monomial space. (For an explicit construction of finite monomial spaces that converge to this subspace, see \cite{amHI}.)

Our goal is to give a real analysis proof of  Theorem \ref{thmis}.
To do this, we first need some preliminary results.
In Section \ref{secb} we state two theorems about Hilbert spaces that we will use. The first, due to von Neumann in 1929, describes
isometries on a Hilbert space. The second, due to Quiggin in 1993, gives a sufficient condition to extend partially defined multipliers of a reproducing kernel Hilbert space without increasing the norm.
We apply Quiggin's theorem to the commutant of the Hardy operator in Section \ref{secd}.

In Section \ref{secc} we describe the Laguerre basis for $L^2$, the basis obtained by evaluating the
Laguerre polynomials on $\log \frac{1}{x}$, which are also the functions obtained by applying
$(1-H^*)^n$ to the constant function $1$.
In section \ref{sece} we deal with multiplicity;
this corresponds to generalizing the notion of finite monomial space to allow not just monomials $x^s$, but
also functions of the form $(\log x)^m x^s$.
In Section \ref{seccyc} we prove that certain rational functions are cyclic for $H^*$.
Finally  in Section \ref{secg} we prove  Theorem \ref{thmis}. Our strategy to prove that an invariant subspace $\m$ of $H$ is a monomial space is to look at the projection $\eta$  of the constant function $1$ onto
$\m^\perp$, and show that the function  $\eta$ uniquely characterizes $\m$. We then approximate
$\eta$ by functions that arise in a similar way from finite monomial spaces, and show that this 
proves that the finite monomial spaces coverge to $\m$ in the sense of Definition \ref{defin1}.

\section{Some results from operator theory}
\label{secb}

An operator $V$ defined on a Hilbert space $\h$ is called an isometry if it preserves norms; a co-isometry is 
the adjoint of an isometry. An isometry $V$ is called pure if $\cap_{n=0}^\infty {\rm ran} (V^*)^n = 0$.
The von Neumann-Wold decomposition describes the structure of isometeries \cite{vn29, wo38}.
We state it not in its most general form, but in a way that will be useful below.
\bt
\label{thmvnw} (von Neumann - Wold)

(i) Every isometry is the direct sum of a unitary operator and a pure isometry.

(ii) If $V$ is a pure isometry on the space $\h$, and $\M = {\rm ker} \ V^*$, then $\h = \vee \{ T^j m : m \in \M\}$.
The dimsnsion of $\M$ is called the multiplicity of $V$.

(iii) If $V$ is a pure isometry of multiplicity $1$ and $f$ is any non-zero vector in $\h$ then
\[
\h \= \vee \{ V^i f, (V^*)^j f : i,j \geq 0 \} .
\]
\et

We shall also need a result on extending the adjoints of multiplication operators, due to Quiggin 
\cite{qui93}. We say that a sesquilinear form $\ell(x,y)$ has one positive square if for any finite set of points
$\{ x_1, \dots, x_N\}$, the self-adjoint $N$-by-$N$ matrix $\ell(x_i,x_j)$ has one positive eigenvalue.
\bt (Quiggin):
\label{thmmq}
Let $(\h,k)$ be a reproducing kernel Hilbert space on a set $X$.
A  sufficient condition that every bounded operator $T$ defined on
$\vee \{ k_x : x \in X_0\}$ for some subset $X_0 \subseteq X$ that has the form
\[
T k_x  \= \alpha(x) k_x, \qquad x \in X_0 \]
extend to a bounded operator $\widetilde{T} : \h \to \h$ that
has the form
\[
\widetilde{T} k_x  \= \widetilde{\alpha}(x) k_x, \qquad x \in X_0 \]
and satisfies $\| \widetilde{T} \| = \| T \|$ is that the reciprocal $\frac{1}{k(x,y)}$ has exactly one positive square.
\et
In the form stated, the converse to Quiggin's theorem is not true. However, if one requires norm-preserving extensions in the vector-valued case too, then the condition that $\frac{1}{k(x,y)}$ has  one positive square is both necessary and sufficient.
This was proved by McCullough \cite{mccul94} in a different context, and put in a unified context in \cite{agmc_cnp}. See also
the paper by Knese \cite{kn20} for an elegant proof of necessity,
 and \cite{ampi}  for a discussion in a book.

\section{The Laguerre basis for $\ltwo$}
\label{secc}

The following identity is a special case of one in \cite{kmp00}.
In our case, it is easily proved by checking on polynomials; see e.g. \cite{amHI}.
\begin{lemma}
\label{lemon1}
Let $f \in \ltwo$. Then
\[
\norm{f}^2 = \norm{(1-H)f}^2 + |\int_0^1 f(x) dx|^2
\]
\end{lemma}
Consequently, $1-H$ is a co-isometry
with one dimensional kernel. As 
\[
(1-H)^k x^n \= \left( \frac{n}{n+1} \right)^k x^n ,
\]
we see that $1-H^*$ is a pure isometry of multiplicity $1$.
Let us state this for future use.
\begin{prop} (Brown, Halmos, Shields) The operator $(1-H^*)$ is a  pure isometry of multiplicity one.
\label{prbas1}
\end{prop}
Proposition \ref{prbas1} was first proved in \cite{bhs65}. 
If we apply powers of $(1-H^*)$ to the constant function $1$, we get a useful orthonormal basis. 
This was first found explicitly in \cite{kmp06}.
\begin{lem}\label{inv.lem.10a}
\be\label{inv.10a}
(H^*)^j \ 1 = (-1)^j \frac{(\log x)^j}{j!}
\ee
\end{lem}
\begin{proof}
We proceed by induction. Clearly, \eqref{inv.10a} holds when $j=0$. Assume $j \ge 0$ and \eqref{inv.10a} holds. Then
\begin{align*}
(H^*)^{j+1} \ 1 &= H^* ((H^*)^j \ 1)\\ \\
&= \frac{(-1)^j}{j!} H^*(\log x)^j\\\\
&=\frac{(-1)^j}{j!}\int_x^1 \frac{(\log t)^j}{t} dt\\ \\
&=\frac{(-1)^j}{j!}\int_{\log x}^0 u^j du\\ \\
&=(-1)^{j+1} \frac{(\log x)^{j+1}}{(j+1)!}
\end{align*}
\end{proof}
\begin{lem}\label{inv.lem.20a}
\[
(1-H^*)^n\  1 \=  \sum_{j=0}^n \binom{n}{j}\frac{(\log x)^j}{j!}
\]
\end{lem}
\begin{proof} 
By  Lemma \ref{inv.lem.10a},
\begin{align*}
(1-H^*)^n\  1 &= \sum_{j=0}^n (-1)^j\binom{n}{j}\ (H^*)^j\ 1\\ \\
&=\sum_{j=0}^n (-1)^j\binom{n}{j} \big((-1)^j \frac{(\log x)^j}{j!}\big)\\ \\
&=\sum_{j=0}^n \binom{n}{j}\frac{(\log x)^j}{j!}.
\end{align*}
\end{proof}
We have proved that the functions
\be
\label{eqbec1}
e_n(x) \= \sum_{j=0}^n \binom{n}{j}\frac{(\log x)^j}{j!}
\ee
are orthonormal. To see that they are complete, note that their closed linear span $\m$  is invariant under
$H^*$ and contains the  function $1$.  Since the constant functions are the kernel of the pure co-isometry
$(1-H)$, this means $\m = \ltwo$ by the von Neumann-Wold Theorem \ref{thmvnw}. So we have proved the following result, which was first proved in \cite{kmp06}.
\bt
\label{thmbec1} (Kaiblinger, Maligranda, Persson)
The functions $e_n$ defined by \eqref{eqbec1} form an orthonormal basis for $\ltwo$.
\et

The Laguerre polynomials are the polynomials 
\[
p_n(x) \=  \sum_{j=0}^n (-1)^j \binom{n}{j}\frac{( x)^j}{j!} .
\]
These are orthogonal polynomials for $L^2[0,\i)$ with the weight function $e^{-x}$.
As $e_n(x) = p_n(\log \frac 1x)$, the change of variables $ t = \log \frac 1x$ is an alternative way to prove that $e_n$
are orthonormal.

The functions $e_n$ are generalized eigenvectors of $H$ at $1$. 
 Later we shall need the following.
\begin{prop}
\label{prge}
Let $s \in \s$. The $(n+1)^{\rm st}$ generalized eigenvector of $H$ with eigenvalue $\frac{1}{s+1}$ is in the linear span
of $\{ x^s, (\log x) x^s, \dots (\log x)^n x^s \}$.
\end{prop}
\bp
We want to prove
\be
\label{eqge3}
{\rm Ker} (H- \frac{1}{s+1})^{n+1} \= \vee \{ x^s, (\log x) x^s, \dots (\log x)^n x^s \}.
\ee
This is true when $n=0$, since
\be
\label{eqge1}
H x^s \= \frac{1}{s+1} x^s .
\ee
Differentiate both sides of \eqref{eqge1} with respect to $s$.
We get
\be
\label{eqge2}
H ( \log x) x^s \= \frac{1}{s+1} (\log x) x^s  - \frac{1}{(s+1)^2} x^s .
\ee
Now we proceed by induction. The inductive hypothesis is that
\be
\label{eqge4}
H (\log x)^n x^s \=  \frac{1}{s+1} (\log x)^n  x^s + \sum_{j=0}^{n-1} c_j(s) (\log x)^j x^s 
\ee
for some functions $c_j$. We have proved \eqref{eqge4} for $n=0$ and $1$. (The $n=1$ case we proved just for expositional clarity).  Assume the hypothesis holds up to $n$.
Differentiate \eqref{eqge4} with respect to $s$ and we get
\be
\label{eqge5}
\notag
H (\log x)^{n+1}  x^s \=  \frac{1}{s+1} (\log x)^{n+1}  x^s - \frac{1}{(s+1)^2}  (\log x)^n  x^s + \sum_{j=0}^{n-1} c_j^\prime(s) (\log x)^j x^s  + c_j(s) (\log x)^{j+1} x^s.
\ee
Thus by induction, \eqref{eqge4} holds for all $n$, and hence so does \eqref{eqge3}.
\ep

\section{Commutant Lifting for the Hardy operator}
\label{secd}

Suppose $T : \ltwo \to \ltwo$ commutes with $H$. Then it must have the same eigenvectors, and so be a 
monomial operator of the form
\be
\label{eqcl1}
T: x^s \ \mapsto \ \alpha(s) x^s .
\ee
When is such an operator bounded?
\bt
\label{thmcl1}
The operator $T$ commutes with $H$ and has norm at most $M$ if and only if $T$ is of the form
\eqref{eqcl1} and, for any finite set $\{ s_i \}_{i=1}^N \subset \s$, the matrix
\be
\label{eqcl2}
\left( \frac{M^2 - \overline{\alpha(s_i)} \alpha(s_j)}{ 1 + \overline{s_i} + s_j } \right)_{i,j=1}^N
\ee
is positive semidefinite.
\et
$T$ may be defined by \eqref{eqcl1} just on some subspace of $\ltwo$. The positivity 
of \eqref{eqcl2} on this set is necessary and  sufficient to lift $T$ from the span of $\{ x^{s_i} \}$ to an operator on all of $\ltwo$ that commutes with $T$ and has the same norm. Without loss of generality we can take $M=1$.
\bt
\label{thmcl2}
Suppose that for some subset $\s_0 \subseteq \s$ there is an operator
\beq
T : \vee \{  x^{s} : s \in \s_0 \} & \ \to \ & \vee \{  x^{s}  : s \in \s_0 \}  \\
T: x^{s} & \mapsto & \alpha(s) x^{s} .
\eeq
A necessary and sufficient condition for $T$ to extend to an operator from $\ltwo$ to $\ltwo$ that commutes with $H$ and has norm at most one is that for every finite set $\{ s_i \} \subseteq \s_0$, we have
\[
\left( \frac{1 - \overline{\alpha(s_i)} \alpha(s_j)}{ 1 + \overline{s_i} + s_j } \right)
\ \geq \ 0.
\]
\et

Notice that Theorem  \ref{thmcl1} is a special case
of Theorem \ref{thmcl2}, so we shall just prove the latter theorem.

\bp {\em (of Theorem \ref{thmcl2}.)}
Necessity: We have that $1 - T^*T \geq 0 $.
Therefore
\be
\label{eqcl3}
\la (1 - T^*T) x^{s_j} , x^{s_i}  \ra \=
\left( \frac{1 - \overline{\alpha(s_i)} \alpha(s_j)}{ 1 + \overline{s_i} + s_j } \right)
\ee
is a positive semi-definite matrix for any subset of $\s_0$.

Sufficiency: Suppose that \eqref{eqcl3} is positive semi-definite for every finite subset of $\s_0$.
Then $T$ commutes with $H|_{\vee \{ x^s : s \in \s_0 \} } $.
Let us define a kernel on $\s$ by
\beq
k(s, t) &\=& \int_0^1 x^t \overline{x^s} dx \\
&=&
 \frac{1}{1 + t + \overline{s}} .
 \eeq
 The reciprocal of $k$ is the sesquilinear form
 \beq
 \ell(s,t) & \= & (\frac 12 + t) + \overline{(\frac 12 + s)} \\
 &=& \frac 12 ( \frac 32 + \bar s) ( \frac 32 + t) - \frac 12 (\frac 12 - \bar s) ( \frac 12 - t) .
 \eeq
 So for any $N \geq 2$ the matrix
 $[\ell(s_i, s_j)]_{i,j=1}^N$ is a rank 2 symmetric matrix, with one positive and
 one negative eigenvalue. By
Theorem \ref{thmmq}, $T$ extends to an operator of norm $1$ on all of $\ltwo$ that has each $x^s$ as an eigenvector, and hence commutes with $H$.
 \ep

\section{Monomial spaces with multiplicity}
\label{sece}
If one takes the two dimensional monomial spaces $\m( s, s+ h)$ and lets $h \to 0$, the
spaces converge to the two-dimensional space spanned by $x^s$ and $\frac{\partial}{\partial s} x^s = (\log x) x^s$.
So if we have a multi-set $S = \{s_1, \dots, s_1, s_2, \dots, s_2,\dots , s_n \}$, where each $s_j$ appears $m_j$ times,
we will define 
\be
\label{eqgms}
\m(S) \= \vee \{ x^{s_1}, (\log x) x^{s_1}, \dots, (\log x)^{m_1 -1} x^{s_1}, \dots ,
x^{s_n}, (\log x) x^{s_n}, \dots, (\log x)^{m_n -1} x^{s_n } \} .
\ee
We shall call a set of the form \eqref{eqgms} a generalized finite monomial space.

\begin{prop}
Every generalized finite monomial space is a limit of finite monomial spaces.
\label{prgms}
\end{prop}
\bp
Fix $m \geq 2$.
Let 
\[
\m_1 \= \vee \{ x^s, (\log x) x^{s}, \dots, (\log x)^{m -1} x^{s} \}.
\]
Let $\omega$ be a primitive $m^{\rm th}$ root of unity, and let $h$ be a small positive number.
Let
\[
\m_2 \= \vee \{ x^{s + \omega^j h} : 0 \leq j \leq m -1 \} .
\]
We shall prove that there is a constant $C$, which depends on $s$ and $m$ but not $h$, so that
\se\att
\begin{eqnarray}
\label{eqgms1}
f \in \m_1 & \ \Rightarrow \ & {\rm dist}(f, \m_2) \leq C h^m \\
f \in \m_2 & \ \Rightarrow \ & {\rm dist}(f, \m_1) \leq C h .
\label{eqgms2}
\end{eqnarray}
\att
As every generalized monomial space of the form \eqref{eqgms} is the sum of finitely many spaces of the form $\m_1$, this will prove the proposition.

In the proof we shall use $C$ for a constant that depends on $m$ but not $h$, and which may change from one line to the next.

Proof of \eqref{eqgms1}. 
(i) First take $s=0$.
By Taylor's theorem, for any unimodular number $\tau$ and any $x >0$ we have
\be
\label{eqgms3}
|x^{\tau h} - \sum_{n=0}^{m-1} \frac{(\tau h)^n}{n!} (\log x)^{n-1}|
\ \leq \ 
\frac{h^m}{m!} (\log x)^m x^{-h} .
\ee
Consider the function $f(x) = (\log x)^n$, for some $n \leq m-1$.
We shall approximate this by the function $g \in \m_2$ given by
\[
g(x) \= \frac{1}{m} \frac{n!}{h^n} \sum_{j=0}^{m-1} \bar \omega^{nj} x^{\omega^j h} .
\]
The choice of arguments for the coefficients means that if one adds together the Taylor series for each $x^{\omega^j h}$, all the terms  cancel except for the ones that are $n \ {\rm mod}\ m$ one, so 
\be
\label{eqgms4}
| g(x) -  (\log x)^n | \ \leq \ C h^{m}  (\log x)^{m+n}  x^{-h} 
\ee
where $C$  is independent of $x$. 
Integrating the square of \eqref{eqgms4} we get that 
${\rm dist}((\log x)^n , \m_2) \leq C h^{m}$.
As the functions $(\log x)^n$ form a basis for $\m_1$, we deduce that \eqref{eqgms1} holds.

(ii) For general $s$, the above argument shows that for each function $x^s (\log x)^n$ there is a function
$g$ in $\m_2$ that satisfies the pointwise estimate
\[
| g(x) -  x^s (\log x)^n  | \ \leq \ C h^{m}  (\log x)^{m+n}  x^{\Re s -h} .
\]
As long as $h$ is small enough that $\Re s - h > - \frac 12$, we again can deduce \eqref{eqgms1}.

\vs
Proof of \eqref{eqgms2}. (i) First take $s=0$.
From \eqref{eqgms3}, we get that ${\rm dist}(x^{\omega^j h}, \m_1) \leq C h^m$.
So the result will follow if we prove that whenever 
$\sum c_j x^{\omega^j h} $ is in the unit ball of $\m_2$, then
$c_j = O(\frac{1}{h^{m-1}})$. This in turn will follow if we can show that
\be
\label{eqgms5}
{\rm dist}(x^{\omega^\ell h}, \vee \{ x^{\omega^i h} : 0 \leq i \leq m-1, \ i \neq \ell \}) \ \geq \ C h^{m-1} 
\ee
for some non-zero $C$, as this proves that the functions $x^{\omega^i h}$ are not too colinear.
For definiteness, we will prove \eqref{eqgms5} for $\ell=0$.
Let $G(i,j) $ denote the Gram matrix with $(i,j)$ entry $\la x^{\omega^i h}, x^{\omega^j h} \ra = \frac{1}{ 1 + \omega^i h + \bar \omega^j h}$. Then
\be
\label{eqgms6}
{\rm dist} ( x^h, \vee_{ 1 \leq i \leq m-1} \{ x^{\omega^i h}  \})^2 \= \det G(i,j)_{i,j=0}^{m-1} / 
\det G(i,j)_{i,j=1}^{m-1}.
\ee
By Cauchy's formula for determinants 
\[
\det ( \frac{1}{ 1 + \omega^i h + \bar \omega^j h}) \= \frac{\prod_{j < i} |\omega^i h - \bar \omega^j h|^2}
{\prod_{i,j} (1 + \omega^i h + \bar \omega^j h) } .
\]
Putting this into \eqref{eqgms6}, we get 
\[
{\rm dist} ( x^h, \vee_{ 1 \leq i \leq m-1} \{ x^{\omega^i h}  \})^2 \=
\frac{ h^{2m-2} \prod_{i=1}^{m-1} |\omega^i - 1|^2}{(1+2h) \prod_{i=1}^{m-1} |1 + (1 + \omega^i) h |^2}.
\]
This equation yields \eqref{eqgms5} for $\ell = 0$, and by symmetry for all $\ell$.

(ii) For general $s \in \s$, a similar argument gives ${\rm dist}(x^{s + \omega^\ell h}, \m_1) \leq C h^m$, and
\[
{\rm dist} ( x^{s+h}, \vee_{ 1 \leq i \leq m-1} \{ x^{s+ \omega^i h}  \})^2 \=
\frac{ h^{2m-2} \prod_{i=1}^{m-1} |\omega^i - 1|^2}{(1+2\Re s + 2h) \prod_{i=1}^{m-1} |1 + 2\Re s + (1 + \omega^i) h |^2}.
\]
\ep
With more work, one can improve \eqref{eqgms2} to $O(h^m)$, but we do not need a sharper estimate.

\begin{corollary}
\label{corgms}
Any space that is a limit of generalized finite monomial spaces is a monomial space. 
\end{corollary}

\section{Some cyclic vectors for $H^*$}
\label{seccyc}

We know from Proposition \ref{prbas1} that the spectrum of $H$ is $\overline{\D(1,1)}$, and for $\lambda \in \D(1,1)$ that 
$H - \lambda$ is Fredholm with index 1.
It follows that $1+ sH$ and $1+ s H^*$ are invertible if and only if $s \in 
\s$.

\begin{lem} 
\label{lemcy1}
If $s \in \s$, then 
\[
x^s \= (1 + s H^*)^{-1} 1 .
\]
\end{lem}
\bp
We have
\beq
\la (1 + s H^*) x^s , x^t \ra &\=& 
\la x^s , ( 1 + \bar s \frac{1}{t+1} x^t \ra \\
&=& \frac{1}{\bar t + 1} \\
&=& \la 1, x^t \ra .
\eeq
\ep

\begin{lem}
\label{lemcy2}
Suppose $f(x) = \sum_{j=0}^N c_j x^{s_j}$, where each $s_j \in \s$. If $f$ is not orthogonal to any monomial $x^t$ for $t \in \s$, then 
$f$ is cyclic for $H^*$.
\end{lem}
\bp
By Lemma \ref{lemcy1}, we have
\[
f(x) \= \sum_{j=0}^N c_j (1 + s_j H^*)^{-1} 1.
\]
Define a rational function $r(z)$ by
\[
r(z) \= \sum_{j=0}^N c_j \frac{1}{1 + s_j z} ,
\]
and let $p,q$ be polynomials with no common factors and $r = p/q$.
The zeroes of $q$ are at the points $\{ - \frac{1}{s_j} : 1 \leq j \leq N \}$.
We have
\be
\label{eqcy1}
f \= p(H^*) q (H^*)^{-1} 1.
\ee

Claim: $p$ has no roots in $\D(1,1)$.

Indeed, 
suppose $p(z_0) = 0$ for some $z_0 \in \D(1,1)$. 
Let $t_0 = \frac{1-\bar z_0}{\bar z_0} \in \s$.
Factor $p$ as $p(z) = (z - z_0) \tilde p (z)$.
Then
\beq
\la f , x^{t_0} \ra &\=& 
\la (H^* - z_0) \tilde p (H^*) q (H^*)^{-1} 1,  x^{t_0} \ra \\
&=& \la \tilde p (H^*) q (H^*)^{-1} 1,  (\frac{1}{t_0 +1} - \bar z_0 ) x^{t_0} \ra \\
&=& 0.
\eeq
This would contradict the assumption that $\la f , x^t \ra \neq 0 $ for all $t \in \s$.

Since $q(H^*)$ is invertible, $f$ is cyclic if and only if $p(H^*) 1$ is cyclic.
We  now factor $p(z) = c\prod (z - z_j)$. 
If $z_j \notin \overline{\D(1,1)}$, then $(H^* - z_j)$ is invertible.
If $z_j \in \partial \D(1,1)$, then $(H^* - z_j)$ has dense range, since $H$ has no eigenvectors on $\partial \D(1,1)$.
Therefore $p(H^*)$ has dense range, and in particular takes cyclic vectors to cyclic vectors.
\ep

If $\la f, x^t \ra =0$, then $f$ is in the range of $H^* - \frac{1}{1+ \bar t}$.
We shall say that $\la f , x^t \ra$ vanishes to order $m$ if $f$ is orthogonal to
$\{ x^t, (\log x) x^t, \dots, (\log x)^{m-1} x^t \}$.
\begin{lem}
\label{lemcy3}
Suppose $f(x) = \sum_{j=0}^N c_j x^{s_j}$, where each $s_j \in \s$, and $f \neq 0$.
Let \[
 \{ t  \in \s : \la f , x^t \ra = 0\} \=  \{ t_1, \dots , t_m\} ,
\]
counted with multiplicity.
Let $z_i = \frac{1}{1+ \bar t_i}$ for $1 \leq i \leq m$.
Then
\be
\label{eqcy2}
f \= \prod_{i=1}^m (H^* - z_i) g ,
\ee
where $g$ is cyclic for $H^*$.
\end{lem}
\bp
Write $f = p(H^*) q(H^*)^{-1} 1$ as in \eqref{eqcy1}.
Let $p^\cup(z) := \overline{p(\bar z)}$.
Then
\beq
\la f , x^t \ra &\=& \la  p(H^*) q(H^*)^{-1} 1, x^t \ra \\
&=& \la 1,  p^\cup(H) q^\cup (H)^{-1} 1 \ra \\
&=& \la 1, \frac{p^\cup ( \frac{1}{t+1})}{q^\cup ( \frac{1}{t+1})} x^t \ra\\
&\=&  \frac{p ( \frac{1}{\bar t+1})}{q ( \frac{1}{\bar t+1})} \la 1, x^t \ra .
\eeq
So the roots of $p$ that lie in $\D(1,1)$  are exactly the points $\{ z_i : 1 \leq i \leq m \}$. (Multiplicity is handled by
Proposition \ref{prge}). 
Factor $p$ as $p(z) = \prod_{i=1}^m (z - z_i) \tilde{p}(z)$, where $\tilde{p}$ has no roots in $\D(1,1)$.
Let $g = \tilde{p}(H^*) q(H^*)^{-1} 1$. Then $g$ is cyclic, and 
\eqref{eqcy2} holds.
\ep

Later we will need the next lemma.
\begin{lem}
\label{lemcy4}
Let $z \in \D(1,1)$. Then 
\[
(H^* - z) [ ( \bar z -1) H^* - \bar z ]^{-1} 
\]
is an isometry.
\end{lem}
\bp
This follows by calculation, using the fact that $1 - H^*$ is isometric.
\ep

\section{Proof of Theorem \ref{thmis}}
\label{secg}

Sufficiency is obvious. For necessity, 
let $\m$ be a proper closed subspace of $\ltwo$ that is invariant for $H$. We must show that it is a monomial space.
\begin{lem}
\label{lem51}
Let $\m$ be a finite dimensional subspace of $\ltwo$, of dimension $n+1$,  that is invariant for $H$.
Then $\m$ is a generalized finite monomial  space, i.e. there exist $n+1$ points $s_0, \dots, s_n$, with multiplicity allowed, 
so that $\m = \vee \{ x^{s_i} : 0 \leq i \leq n \}$.
\end{lem}
\bp
Consider $H |_\m$, which leaves $\m$ invariant. The space $\m$ is spanned by the eigenvectors and generalized eigenvectors of $H$ that lie in $\m$.
Suppose the corresponding eigenvalues are $s_j$, with multiplicity $m_j$.
 By Proposition \ref{prge}, the generalized eigenvectors  are of the form
$x^{s_j}, (\log x) x^{s_j}, \dots, (\log x)^{m_j -1} x^{s_j}$. Therefore $\m$ is the generalized finite monomial 
space corresponding to the exponents $s_j$ with multiplicity $m_j$.
\ep
To prove the full theorem, we use the idea of wandering subspace, due to Halmos \cite{hal61}.
Let 
\[
k_0 \ := \ \min \{ k : e_k \notin \m \} .\]
Write $\n$ for $\m^\perp$.
Write $e_{k_0}  = \xi +\eta$, where $\xi \in \m$ and $\eta \in \n$.
The assumption that $ e_{k_0} \notin \m$ means $\eta \neq 0$. 
Let $u = \frac{\eta}{\|\eta\|}$.

\begin{lem}
\label{lem52}
We have $u \perp (1-H^*) \n$.
\end{lem}
\bp
Let $f \in \n$. Then
\beq
\la u, (1 - H^*) f \ra &\=& \la \|\eta \|\  e_{k_0} ,  (1 - H^*) f \ra \\
&=& \|\eta \| \la (1-H) e_{k_0}, f \ra.
\eeq
If $k_0 = 0$, then $(1-H)  e_{k_0} = 0$. If $k_0 > 0$, then $(1-H)  e_{k_0} = e_{k_0 - 1} \in \m$.
Either way, the inner product with $f$ is $0$.
\ep

Define an operator $R$ in terms of the orthonormal basis $e_n$ from \eqref{eqbec1} by
\be
\label{eqbe1}
R: e_n \ \mapsto \ (1-H^*)^n u .
\ee
\begin{lem}
\label{lem53}
 The operator $R$ defined by \eqref{eqbe1} is an isometry from $\ltwo$ onto $\n$.
\end{lem}
\bp
The functions $\{ (1-H^*)^n u : n \geq 0 \}$ form an orthonormal set.
Indeed, by Proposition \ref{prbas1} and Lemma \ref{lem52}, if $m \geq n$ then
\beq
\la (1-H^*)^m u, (1-H^*)^n u \ra &\=& \la (1-H^*)^{m-n} u,  u \ra \\
&=& \delta_{m,n} .
\eeq
As $R$ maps an orthonormal basis to an orthonormal set, it must be an isometry onto its range.

We know that the range of $R$ is contained in $\n$. To see that it is all of $\n$, observe that by
Lemma \ref{lem52}, we have that
\[
\vee \{ (1-H)^m u : m \geq 1 \}
\]
is contained in $\n^\perp = \m$. As $1-H$ is a pure isometry of multiplicity $1$, by Theorem \ref{thmvnw} for any non-zero vector $f$ the vectors
\[
\{ (1-H)^m f, (1-H^*)^n f : m,n \geq 0 \}
\]
span $\ltwo$. Therefore in particular, 
$\vee \{ (H^*)^n u : n \geq 0 \}$ and $ \vee \{ H^m u : m \geq 1 \}$ span $\ltwo$, so 
\beq
\n &\= & \vee \{ (H^*)^n u : n \geq 0 \} \\
\m &=& \vee \{ H^m u : m \geq 1 \} .
\eeq
\ep
Let us calculate $T = R^*$, the adjoint of $R$.
\begin{lem}
\label{lem54}
The adjoint of $R$ is given by the operator
\be
\label{eqbe2}
T : x^s \ \mapsto \ (1+s) \la x^s, u \ra x^s .
\ee
\end{lem}
\bp
We have
\beq
\la R^* x^s, e_n \ra &\=& \la x^s , (1-H^*)^n u \ra
\\
&=& 
\la (1-H)^n x^s, u \ra \\
&=&
\left( \frac{s}{s+1} \right)^n \la x^s, u \ra .
\eeq
We also have by Lemma \ref{inv.lem.20a}
\beq
\la T x^s, e_n \ra &\=&  (1+s) \la x^s, u \ra \la x^s , (1-H^*)^n 1 \ra
\\
&=& 
(1+s) \la x^s, u \ra \la (1-H)^n  x^s , 1 \ra \\
&=&
\left( \frac{s}{s+1} \right)^n \la x^s, u \ra .
\eeq
Therefore $T = R^*$.
\ep
We want to approximate $\m$ by monomial spaces. We shall do this by approximating $u$ by linear combinations of monomials.
We have proved that $T$ is a co-isometry that commutes with $H$.
This means by Theorem \ref{thmcl2} that for each $N$, the matrix
\be
\label{eqbe3}
\notag
\left( \frac{1 - (i+1)(j+1)\la u, x^i \ra \la x^j, u \ra}{1 + i + j} \right)_{i,j=0}^N \ \geq \ 0 .
\ee
We shall assume for the remainder of this section  that $N$ is large enough that $\la u, x^i \ra \neq 0$ for some $i \leq N$. 
Let $C_N \geq 1$ be the largest number $C$ so that
\be
\label{eqbe4}
\notag
\left( \frac{1 - C^2 (i+1)(j+1)\la u, x^i \ra \la x^j, u \ra}{1 + i + j} \right)_{i,j=0}^N \ \geq \ 0 .
\ee
The hypothesis on $N$ means $C_N$ is finite, and $\lim_{N \to \infty} C_N = 1$.
Define $\tilde T_N$ by
\[
\tilde T_N : x^i \  \mapsto \ C_N (i+1) \la x^i , u \ra x^i, \qquad 0 \leq i \leq N .
\]
By Theorem \ref{thmcl2}, this extends to an operator  $T_N$ that maps $\ltwo$ to $\ltwo$, commutes with $H$, and has norm equal to $1$.
So $T_N$ is of the form 
\be
\label{eqbe71}
T_N : x^s \ \mapsto\  \alpha_N (s) x^s.
\ee
\begin{lem}
\label{lem55}
The function $\alpha_N(s)$ is a rational function of degree at most $N$,  and maps $\s$ to $\D$.
\end{lem}
\bp
We know that 
\be
\label{eqbe6}
\alpha_N(i) = C_N (i+1) \la u , x^i \ra, \qquad 0 \leq i \leq N,
\ee
Let $\gamma$ be a non-zero vector in the kernel of
\be
\label{eqbe5}
\notag
\left( \frac{1 - C_N^2 (i+1)(j+1)\la u, x^i \ra \la x^j, u \ra}{1 + i + j} \right)_{i,j=0}^N \ \geq \ 0 .
\ee
By Theorem \ref{thmcl1}, the matrix \eqref{eqcl2} has to be positive semidefinite when we augment the
set $\{ 0, \dots, N \}$ by any other point $s$. This means by Lemma \ref{lem56} that the first $N+1$ entries in the
 last column of the extended $(N+2)$-by-$(N+2)$ matrix must be orthogonal to $\gamma$, so
 \[
 \sum_{i=0}^N \frac{ 1 - \overline{\alpha_N (i)} \alpha_N(s)}{1 + i + s} \gamma_i \= 0 .
 \]
This equation yields
\be
\label{eqbe7}
\left(  \sum_{i=0}^N  \frac{  \overline{\alpha_N (i)} \gamma_i}{1 + i + s} \right)
\alpha_N(s)
\=
 \sum_{i=0}^N  \frac{   \gamma_i}{1 + i + s} 
\ee
Let $R(s)$ denote 
 the right-hand side of \eqref{eqbe7},  and  $L(s)$ denote the coefficient of $\alpha_N(s)$ on the left.
 Both $R$ and $L$ are 
rational functions, vanishing at infinity, with simple poles exactly in the set 
\[
\{ -1 - i : \gamma_i \neq 0 \} .
\]
Their ratio $\alpha_N = R/L$, therefore, is a rational function with poles at the zero set of $L$, and zeroes on the zero set of 
$R$. The degree will be at most $N$, since they both have zeroes at infinity.

As $\| T x^s \| = |\alpha_N (s) | \| x^s \| \leq \| x^s \|$, we have $\alpha_N : \s \to \D$. 
\ep 

We used the following lemma, whose proof is elementary linear algebra.
\begin{lem}
\label{lem56}
Suppose $A$ is a positive semi-definite matrix, and $\gamma$ is a non-zero vector in the kernel of $A$.
If there is a vector $\beta$ and a constant $c$ so that
\[
\begin{pmatrix}
A & \beta \\
\beta^* & c 
\end{pmatrix}
\ \geq \ 0 ,
\]
then $\la \beta, \gamma \ra = 0 $.
\end{lem}


\begin{lem}
\label{lem57}
Let $\alpha$ be a rational function of degree $N$ with all its poles  in the set $\{ s : \Re(s) <  - \frac{1}{2} \}$,
and with no pole at $\infty$.

(i) If $\alpha(-1) \neq 0$, 
then there exists a sequence $\{  s_0, \dots , s_{N} \}$, with multiplicity allowed, and a function
$u_N$ in the generalized finite monomial  space $\m(\{ s_0, s_1, \dots , s_{N} \})$, so that 
\be
\label{eqbe8}
\alpha(s) \= (1+s) \la x^s , u_N \ra \qquad \forall s \in \s.
\ee
Moreover we can take $s_0 = 0$.

(ii) If $\alpha(-1) = 0$, then there exists a sequence $\{  s_1, \dots , s_{N} \}$, with multiplicity allowed, and a function
$u_N$ in the generalized finite monomial  space $\m(\{  s_1, \dots , s_{N} \})$, so that 
\be
\label{eqbe8.1}
\alpha(s) \= (1+s) \la x^s , u_N \ra \qquad \forall s \in \s.
\ee
\end{lem}
\bp
Expand $\alpha(s)/(s+1)$ by partial fractions to get
\[
\frac{\alpha(s)}{s+1} \=  \sum_{j=1}^p \sum_{r=1}^{m_j} \frac{(r-1)! c^r_j}{(s-\lambda_j)^r }.
\]
There is no constant term, since the left-hand side vanishes at $\infty$.
We can assume that $c_j^{m_j} \neq 0$ for each $j$.
In case (i), 
 there is a  pole, which we denote $\lambda_1$, at $-1$, and 
 $\sum_{j=1}^p m_j = N+1$.
 In case (ii) there is no pole at $-1$, and  $\sum_{j=1}^p m_j = N$.

The inverse Laplace transform of $\alpha(s)/(s+1)$ is 
\[
F(t) \= \sum_{j=1}^p \sum_{r=1}^{m_j}  c^r_j t^{r-1} e^{\lambda_j t} .
\]
Define
\begin{eqnarray}
\label{eqbe21}
\notag
u_N(x) &\=&  \frac{1}{x} \overline{F(\log \frac 1x)}
\\ &=& 
\sum_{j=1}^p \sum_{r=1}^{m_j}  \bar c^r_j ( \log \frac 1x)^{r-1} x^{-\bar \lambda_j -1}  .
\label{eqbe22}
\end{eqnarray}
Then, making the substitution $e^{-t} = x$, we get
\beq
\la x^s , u_N \ra &\=& \int_0^1 x^s \overline{u_N (x)} dx \\
&=& \int_0^\infty e^{-st} {F(t)}  dt \\
&=& ({\mathcal L}  F )(s)\\
&=& \frac{\alpha(s)}{s+1} .
\eeq
Notice that each point $-1-\bar \lambda_j$ is in $\s$.
We now define the multiset $\{s_0,   s_1, \dots , s_{N} \}$  (respectively, 
$\{s_1, \dots , s_{N} \}$ )
by taking $m_j$ copies of the point $-\bar \lambda_j - 1$ for each $j$.
\ep
We shall prove in Lemma \ref{lem59.1} that case (ii) cannot occur for $\alpha_N$.

\begin{lem}
\label{lem60}
Let $\K = {\rm ker}( 1 -  T_N T_N^*)$. Then $\K$ is $H^*$ invariant. 
\end{lem}
\bp
As $H$ commutes with $T_N$ and $(1-H)(1-H^*) = 1$ by Lemma \ref{lemon1}, we have
\[
T_N \= (1-H)T_N (1 - H^*) .
\]
So if $g \in \K$ then
\beq
\| g \|^2 &\= & \| T_N^* g \|^2 \\
&\=& \la (1 - H) T_N^* (1 -H^*) g, T_N^*g \ra .
\eeq
As $\| 1 - H^* \|$ and $\| T_N^*\|$ are both equal to $1$, we have 
\beq
\| T_N^* (1 - H^*) g \| &\=& \| (1-H^*) g \| \\
&\=& \| g \| .
\eeq
Therefore $(1-H^*)g$ is also in $\K$, and hence $\K$ is $H^*$ invariant.
\ep

\begin{lem}
\label{lem58}
The operator $T_N$ is a co-isometry.
\end{lem}
\bp
Let $\gamma$ be as in the proof of Lemma \ref{lem55}. Let
$f(x) \= \sum_{j=0}^N \gamma_j x^j$.
Then $(1- T_N^* T_N) f = 0$, so $T_N$ attains its norm on $f$.
Let
\[
g  \= T_N f  \=  \sum_{j=0}^N \gamma_j \alpha(j)  x^j .
\]
As $f = T_N^* T_N f$, we have $g = T_N T_N^* g$.

To prove  $T_N$ is a co-isometry, we must show that
\[
\K \= {\rm ker}( 1 -  T_N T_N^*) 
\]
is all of $\ltwo$. 
By Lemma \ref{lem60}, we know that $\K$ is $H^*$ invariant, and it contains the polynomial $g$.
If $g$ were not orthogonal to any $x^t$, then we would be done by Lemma \ref{lemcy2}.

As $\la x^t , g \ra$ is a non-zero rational function of $t$, it can only have finitely many zeroes in $\s$;
label these $\{ t_1, \dots, t_m \}$, counting multiplicity. 
By Lemma \ref{lemcy3} we have
\be
\label{eqbe9}
\notag
g \= \prod_{i=1}^m (H^* - z_i) h_1 ,
\ee
where $h_1$ is cyclic for $H^*$ and $z_i = \frac{1}{1 + \bar t_i}$.
Let 
\[
h_2 \= \prod_{i=1}^m [ (\bar z_i -1) H^* - \bar z_i ]\  h_1 .
\]
Then $h_2$ is cyclic since it is an invertible operator applied to $h_1$.
Let 
\[
r(z) \=  \prod_{i=1}^m \frac{ z - z_i} { (\bar z_i -1) z - \bar z_i} .
\]
By Lemma \ref{lemcy4}, $r(H^*)$ is an isometry, and we have $r(H^*) h_2 = g $.
Therefore 
\beq
\| h_2 \| & \= & \| r(H^*) h_2 \| \\
&=& \| T_N T_N^*  r(H^*) h_2 \| \\
&\leq & \|  T_N^*  r(H^*) h_2 \| \\
&=& \|  r(H^*)  T_N^* h_2 \| \\
&=&  \|   T_N^* h_2 \| \\
&\leq & \| h_2 \| .
\eeq
Therefore $h_2 \in \K$, and since $\K$ is $H^*$ invariant and $h_2$ is cyclic, we get
that $\K$ is all of $\ltwo$ and hence $T_N^*$ is an isometry.
\ep
Let $R_N = T_N^*$. A similar calculation to the proof of Lemma \ref{lem54} yields:
\begin{lem}
\label{lem60}
The operator $R_N$ maps $e_n$ to $(1-H^*)^n u_N$.
\end{lem}

We are now ready to define $\m_N$. Let $\alpha_N$ be as in \eqref{eqbe71}.
Apply Lemma \ref{lem57} to $\alpha_N$ to get, in case (i),  
a space  $\m(\{ s_0, s_1, \dots , s_{N} \})$ that contains $u_N$
given by \eqref{eqbe22} and satisfies \eqref{eqbe8},
and in case (ii) a space  $\m(\{  s_1, \dots , s_{N} \})$ that contains $u_N$
given by \eqref{eqbe22} and satisfies \eqref{eqbe8.1}.

We show that Case (ii)  of  Lemma \ref{lem57} cannot occur.
\begin{lem}
\label{lem59.1}
We have $\alpha_N(-1) \neq 0$.
\end{lem}
\bp
Let us assume that we are in Case (ii) of  Lemma \ref{lem57}.
Let $t_j = - \bar\lambda_j - 1$. In the sequence $\{  s_1, \dots, s_N \}$ each $t_j$ appears with multiplicity $m_j$,
and no $t_j$ is $0$.
We have
\[
u_N =  \sum_{j=1}^p \sum_{r=1}^{m_j}  \bar c^r_j (- \log x)^{r-1} x^{t_j} .
\]
 Since $R_N$ is isometric by Lemma \ref{lem58},
we have $u_N$ is orthogonal to $(1-H^*)^k u_N$ for every $k \geq 1$, and hence $u_N$ is also
orthogonal to $(1-H)^k u_N$ for every $k \geq 1$.
For each $j \geq 1$, let $p_j^k$ be a polynomial that vanishes at  $0$, vanishes at $t_i$ to order $m_i$ if $i \neq j$, and
vanishes at $t_j$ to order $k$.
Since each such polynomial vanishes at zero, we have
\be
\label{eqbe24}
\la u_N , p_j^k(1-H)  u_N \ra \= 0.
\ee
Consider 
\[ p_j^{m_j} (1-H) u_N \= \bar  c_j^{m_j} x^{t_j} .
\]
By \eqref{eqbe24}, we conlcude that $u_N \perp x^{t_j}$.
Similarly $p_j^{m_j -1} (1-H) u_N$  equals $\bar c_j^{m_j} (\log x) x^{t_j}$ plus some multiple of $x^{t_j}$.
Therefore we conclude that $u_N$ is also orthogonal to $(\log x) x^{t_j}$. 
Continuing in this way, we conclude that $u_N$ is orthogonal to every function in $\m(\{  s_1, \dots , s_{N} \})$.
Since $u_N$ itself is in this space, we conclude that $u_N = 0$, a contradiction.
\ep

Let $\m_N = \m(\{  s_1, \dots , s_{N} \})$, in other words the space
$\m( \{ s_0,  s_1, \dots , s_{N} \})$ with the multiplicity at $0$ reduced by $1$.
Here is the final step.
\begin{lem}
\label{lem59}
The sequence $\m_N$ tends to $\m$.
\end{lem}
\bp
Let $t_j = - \bar \lambda_j - 1$, with $t_1 = 0$.
We have
\[
u_N =  \sum_{j=1}^p \sum_{r=1}^{m_j}  \bar c^r_j (- \log x)^{r-1} x^{t_j} .
\]
As in the proof of Lemma \ref{lem59.1}, we conclude that $u_N$ is orthogonal to 
$p(1-H) \m(  \{ s_0,  s_1, \dots , s_{N} \})$ for every polynomial $p$ that vanishes at $0$.
So $u_N$ is a constant multiple of the projection of $e_{m_1 -1} $ onto $\m_n^\perp$.

By Lemma \ref{lem53}, $T_N^*$ is an isometry from $\ltwo$ onto $\m_N^\perp$.
Therefore the projection $P_{\m_N}$ onto $\m_N$ is given by  $1 - T_N^* T_N$.
We have 
\beq
T_N : x^i &\ \mapsto \ & (i+1) \la x^i , u_N \ra x^i, \qquad 0 \leq i \leq N \\
&=& C_N 
(i+1) \la x^i , u\ra x^i,  \qquad 0 \leq i \leq N .
\eeq
As $N \to \infty$, we have $u_N \to u$ weakly and so $T_N \to T$ in SOT.
Therefore $P_{\m_N} \to P_{\m} = 1 - T^* T$ in WOT and hence also SOT (since a sequence of projections converges in the SOT if and only if it converges WOT).
\ep

\section{Open Question}

Let $1 < p < \infty$, and $p \neq 2$. The Hardy operator is bounded on $L^p[0,1]$, and has
$x^s$ as an eigenvector whenever $s \in \s_p = \{ s \in \C : \Re (s) > - \frac{1}{p} \}$.
Any space that is the limit of finite monomial spaces (with powers in $\s_p$) is therefore
invariant for $H$. Is every closed subspace of $L^p[0,1]$ that is invariant for $H$ of this form?

\bibliography{../references_uniform_partial}
\end{document}